\documentclass[11pt, a4paper, twoside,leqno]{amsart}
\usepackage[centering, totalwidth = 390pt, totalheight = 580pt]{geometry}
\usepackage{amssymb, amsmath, amsthm}
\usepackage{enumerate, microtype, stmaryrd, url} 
\usepackage[latin1]{inputenc}
\usepackage[arrow, matrix, tips, curve]{xy}
 \usepackage{color}
\definecolor{darkgreen}{rgb}{0,0.45,0}
\usepackage[colorlinks,citecolor=darkgreen,linkcolor=darkgreen]{hyperref}
\SelectTips{cm}{10}

 \DeclareMathOperator{\ob}{ob}

\newcommand{\cat}[1]{\mathbf{#1}}
\newcommand{\op}{\mathrm{op}}

\newcommand{\thg}{{\mathord{\text{--}}}}

\newcommand{\abs}[1]{{\left|{#1}\right|}}

\newcommand{\cd}[2][]{\vcenter{\hbox{\xymatrix#1{#2}}}}

\renewcommand{\phi}{\varphi}
\newcommand{\A}{{\mathcal A}}
\newcommand{\B}{{\mathcal B}}
\newcommand{\C}{{\mathcal C}}
\newcommand{\D}{{\mathcal D}}
\newcommand{\E}{{\mathcal E}}
\newcommand{\F}{{\mathcal F}}

\newcommand{\I}{{\mathcal I}}

\renewcommand{\O}{{\mathcal O}}
\renewcommand{\P}{{\mathcal P}}
\newcommand{\Q}{{\mathcal Q}}
\newcommand{\R}{{\mathcal R}}

\newcommand{\V}{{\mathcal V}}
\newcommand{\W}{{\mathcal W}}

\newcommand{\xtor}[1]{\cdl[@1]{{} \ar[r]|-{\object@{|}}^{#1} & {}}}
\newcommand{\tor}{\ensuremath{\relbar\joinrel\mapstochar\joinrel\rightarrow}}

\makeatletter

\def\hookleftarrowfill@{\arrowfill@\leftarrow\relbar{\relbar\joinrel\rhook}}
\def\twoheadleftarrowfill@{\arrowfill@\twoheadleftarrow\relbar\relbar}
\def\leftbararrowfill@{\arrowdoublefill@{\leftarrow\mkern-5mu}\relbar\mapstochar\relbar\relbar}
\def\Leftbararrowfill@{\arrowdoublefill@{\Leftarrow\mkern-2mu}\Relbar\Mapstochar\Relbar\Relbar}
\def\leftringarrowfill@{\arrowdoublefill@{\leftarrow\mkern-3mu}\relbar{\mkern-3mu\circ\mkern-2mu}\relbar\relbar}
\def\lefttriarrowfill@{\arrowfill@{\mathrel\triangleleft\mkern0.5mu\joinrel\relbar}\relbar\relbar}
\def\Lefttriarrowfill@{\arrowfill@{\mathrel\triangleleft\mkern1mu\joinrel\Relbar}\Relbar\Relbar}

\def\hookrightarrowfill@{\arrowfill@{\lhook\joinrel\relbar}\relbar\rightarrow}
\def\twoheadrightarrowfill@{\arrowfill@\relbar\relbar\twoheadrightarrow}
\def\rightbararrowfill@{\arrowdoublefill@{\relbar\mkern-0.5mu}\relbar\mapstochar\relbar\rightarrow}
\def\Rightbararrowfill@{\arrowdoublefill@{\Relbar\mkern-2mu}\Relbar\Mapstochar\Relbar\Rightarrow}
\def\rightringarrowfill@{\arrowdoublefill@\relbar\relbar{\mkern-2mu\circ\mkern-3mu}\relbar{\mkern-3mu\rightarrow}}
\def\righttriarrowfill@{\arrowfill@\relbar\relbar{\relbar\joinrel\mkern0.5mu\mathrel\triangleright}}
\def\Righttriarrowfill@{\arrowfill@\Relbar\Relbar{\Relbar\joinrel\mkern1mu\mathrel\triangleright}}

\def\leftrightarrowfill@{\arrowfill@\leftarrow\relbar\rightarrow}
\def\mapstofill@{\arrowfill@{\mapstochar\relbar}\relbar\rightarrow}

\renewcommand*\xleftarrow[2][]{\ext@arrow 20{20}0\leftarrowfill@{#1}{#2}}
\providecommand*\xLeftarrow[2][]{\ext@arrow 60{22}0{\Leftarrowfill@}{#1}{#2}}
\providecommand*\xhookleftarrow[2][]{\ext@arrow 10{20}0\hookleftarrowfill@{#1}{#2}}
\providecommand*\xtwoheadleftarrow[2][]{\ext@arrow 60{20}0\twoheadleftarrowfill@{#1}{#2}}
\providecommand*\xleftbararrow[2][]{\ext@arrow 10{22}0\leftbararrowfill@{#1}{#2}}
\providecommand*\xLeftbararrow[2][]{\ext@arrow 50{24}0\Leftbararrowfill@{#1}{#2}}
\providecommand*\xleftringarrow[2][]{\ext@arrow 10{26}0\leftringarrowfill@{#1}{#2}}
\providecommand*\xlefttriarrow[2][]{\ext@arrow 80{24}0\lefttriarrowfill@{#1}{#2}}
\providecommand*\xLefttriarrow[2][]{\ext@arrow 80{24}0\Lefttriarrowfill@{#1}{#2}}

\renewcommand*\xrightarrow[2][]{\ext@arrow 01{20}0\rightarrowfill@{#1}{#2}}
\providecommand*\xRightarrow[2][]{\ext@arrow 04{22}0{\Rightarrowfill@}{#1}{#2}}
\providecommand*\xhookrightarrow[2][]{\ext@arrow 00{20}0\hookrightarrowfill@{#1}{#2}}
\providecommand*\xtwoheadrightarrow[2][]{\ext@arrow 03{20}0\twoheadrightarrowfill@{#1}{#2}}
\providecommand*\xrightbararrow[2][]{\ext@arrow 01{22}0\rightbararrowfill@{#1}{#2}}
\providecommand*\xRightbararrow[2][]{\ext@arrow 04{24}0\Rightbararrowfill@{#1}{#2}}
\providecommand*\xrightringarrow[2][]{\ext@arrow 01{26}0\rightringarrowfill@{#1}{#2}}
\providecommand*\xrighttriarrow[2][]{\ext@arrow 07{24}0\righttriarrowfill@{#1}{#2}}
\providecommand*\xRighttriarrow[2][]{\ext@arrow 07{24}0\Righttriarrowfill@{#1}{#2}}

\providecommand*\xmapsto[2][]{\ext@arrow 01{20}0\mapstofill@{#1}{#2}}
\providecommand*\xleftrightarrow[2][]{\ext@arrow 10{22}0\leftrightarrowfill@{#1}{#2}}
\providecommand*\xLeftrightarrow[2][]{\ext@arrow 10{27}0{\Leftrightarrowfill@}{#1}{#2}}

\makeatother

\newcommand{\twocong}[2][0.5]{\ar@{}[#2] \save ?(#1)*{\cong}\restore}
\newcommand{\twoeq}[2][0.5]{\ar@{}[#2] \save ?(#1)*{=}\restore}
\newcommand{\rtwocell}[3][0.5]{\ar@{}[#2] \ar@{=>}?(#1)+/l 0.2cm/;?(#1)+/r 0.2cm/^{#3}}
\newcommand{\ltwocell}[3][0.5]{\ar@{}[#2] \ar@{=>}?(#1)+/r 0.2cm/;?(#1)+/l 0.2cm/^{#3}}
\newcommand{\ltwocello}[3][0.5]{\ar@{}[#2] \ar@{=>}?(#1)+/r 0.2cm/;?(#1)+/l 0.2cm/_{#3}}
\newcommand{\dtwocell}[3][0.5]{\ar@{}[#2] \ar@{=>}?(#1)+/u  0.2cm/;?(#1)+/d 0.2cm/^{#3}}
\newcommand{\dltwocell}[3][0.5]{\ar@{}[#2] \ar@{=>}?(#1)+/ur  0.2cm/;?(#1)+/dl 0.2cm/^{#3}}
\newcommand{\drtwocell}[3][0.5]{\ar@{}[#2] \ar@{=>}?(#1)+/ul  0.2cm/;?(#1)+/dr 0.2cm/^{#3}}
\newcommand{\dthreecell}[3][0.5]{\ar@{}[#2] \ar@3{->}?(#1)+/u  0.2cm/;?(#1)+/d 0.2cm/^{#3}}
\newcommand{\utwocell}[3][0.5]{\ar@{}[#2] \ar@{=>}?(#1)+/d 0.2cm/;?(#1)+/u 0.2cm/_{#3}}
\newcommand{\dtwocelltarg}[3][0.5]{\ar@{}#2 \ar@{=>}?(#1)+/u  0.2cm/;?(#1)+/d 0.2cm/^{#3}}
\newcommand{\utwocelltarg}[3][0.5]{\ar@{}#2 \ar@{=>}?(#1)+/d  0.2cm/;?(#1)+/u 0.2cm/_{#3}}

\newdir{(}{{}*!<0em,-.14em>-\cir<.14em>{l^r}}
\newdir{ (}{{}*!/-5pt/\dir{(}}
\newdir{ >}{{}*!/-5pt/\dir{>}}

\theoremstyle{definition}

\swapnumbers
\theoremstyle{plain}
\newtheorem{Thm}[subsection]{Theorem}
\newtheorem{Prop}[subsection]{Proposition}
\newtheorem{Cor}[subsection]{Corollary}

\numberwithin{equation}{section}

\theoremstyle{definition}
\newtheorem{Defn}[subsection]{Definition}

\newtheorem{Ex}[subsection]{Example}
\newtheorem{Exs}[subsection]{Examples}
\newtheorem{Rk}[subsection]{Remark}

\begin{document}
\leftmargini=2em
\title{Topological = total}
\author{Richard Garner} 
\address{Department of Computing, Macquarie  University, NSW 2109, Australia} 
\email{richard.garner@mq.edu.au}

\date{\today}

\thanks{This work was supported by the Australian Research Council's
  \emph{Discovery Projects} scheme, grant number DP110102360.}

\begin{abstract}
  A notion of central importance in categorical topology is that of
  topological functor.  A faithful functor $\E \to \B$ is called
  topological if it admits cartesian liftings of all (possibly large)
  families of arrows; the basic example is the forgetful functor
  $\cat{Top} \to \cat{Set}$. A topological functor $\E \to 1$ is the
  same thing as a (large) complete preorder, and the general
  topological functor $\E \to \B$ is intuitively thought of as a
  ``complete preorder relative to $\B$''. We make this intuition
  precise by considering an enrichment base $\Q_\B$ such that
  $\Q_\B$-enriched categories are faithful functors into $\B$, and
  show that, in this context, a faithful functor is topological if and
  only if it is total (=totally cocomplete) in the sense of
  Street--Walters. We also consider the MacNeille completion of a
  faithful functor to a topological one, first described by Herrlich,
  and show that it may be obtained as an instance of Isbell's generalised
  notion of MacNeille completion for enriched categories.
\end{abstract}

\maketitle

\section{Introduction}
\label{sec:intro}

One of the more inconvenient facts in mathematics is that of the
relatively bad behaviour of the category $\cat{Top}$ of topological
spaces: though complete, cocomplete and extensive, it is not regular,
coherent, locally presentable, or (locally) cartesian closed. Many
authors have thus been led to propose replacing the category of spaces
by some other category which either embeds $\cat{Top}$ or embeds into
$\cat{Top}$ in a reasonable manner, but which possesses some of the
desirable properties that $\cat{Top}$ itself lacks; some examples are
the categories of quasitopological spaces, approach spaces,
convergence spaces, uniformity spaces, nearness spaces, filter spaces,
epitopological spaces, Kelley spaces, compact Hausdorff spaces,
$\Delta$-generated spaces, or of sheaves on some small subcategory of
$\cat{Top}$.

In attempting to impose some kind of order on this proliferation of
notions, a useful organising framework is that of \emph{categorical
  topology}~\cite{Herrlich1971Categorical}; a key insight of which is
that categories of space-like structures are most fruitfully studied
not as categories \emph{simpliciter}, but as categories equipped with
a faithful functor $\A \to \cat{Set}$. Desirable properties of a
category of space-like structures can be re-expressed as properties of
this functor, and the process of replacing $\cat{Top}$ by a category
with such properties then becomes that of adjoining those desirable
properties to the usual forgetful functor $\cat{Top} \to \cat{Set}$,
or some subfunctor thereof; see~\cite{Lowen2001Supercategories} for an
overview.

In general, categorical topology studies faithful functors not just
into $\cat{Set}$, but into an arbitrary base category $\B$. When $\B =
1$, such functors are simply preorders; this motivates the step of
regarding a general faithful functor $p \colon \E \to \B$ as a
``preorder relative to $\B$'', and many aspects of categorical
topology can be seen as as elucidation of this idea.  Of particular
importance are the ``complete preorders relative to $\B$'', the
so-called \emph{topological
  functors}~\cite{Antoine1966Extension,Roberts1968A-characterization,Herrlich1974Topological}.
A faithful functor is topological if it admits a generalisation of the
characterising property of a Grothendieck fibration in which one may
form cartesian liftings not just of single arrows, but of arbitrary
(possibly large) families of them.

The intuition that topological functors are relativised complete
preorders can be seen in many places throughout the literature: for
example, in the various completion processes by which a faithful
functor may be turned into a topological
one~\cite{Herrlich1976Initial}, which correspond to the constructions
by which a poset may be turned into a complete lattice, and indeed
reduce to these constructions when $\B = 1$. However, nowhere is this
basic intuition wholly justified. The objective of this paper is to
rectify this by way of enriched category theory: given a category
$\B$, we describe an enrichment base for which the enriched categories
are faithful functors into $\B$, and the enriched categories which are
\emph{total} (=~totally cocomplete)---in a sense to be recalled
below---are the topological functors into $\B$.

It is worth saying a few words about the kind of enrichment base we
will require. Most enrichments, as in~\cite{Kelly1982Basic}, involve a
monoidal category $\V$, with a $\V$-category then having homs which
are objects of $\V$. For example, when $\V$ is the monoidal poset
$(\cat 2, \wedge, \top)$, $\V$-categories are precisely preorders, and
so we may regard order theory as a particular kind of enriched
category theory.  In the early 1980's,
Walters~\cite{Walters1982Sheaves,Walters1981Sheaves} realised the
value of a more general kind of enrichment (first suggested by
B\'enabou~\cite{Benabou1967Introduction}) based on a bicategory $\W$;
in a $\W$-category, each object is typed by an object of $\W$, while the
homs are appropriately-typed \emph{morphisms} of $\W$. This
kind of enrichment includes the more familiar kind on regarding a
monoidal category $\V$ as a one-object bicategory.

Walters' application for this notion was to sheaf theory: for each
site, he describes a bicategory $\W$ such that Cauchy-complete,
symmetric, skeletal $\W$-categories are precisely sheaves on that
site. More generally, posets internal to sheaves correspond to
Cauchy-complete skeletal $\W$-categories\footnote{Although the general
  Cauchy-complete $\W$-category corresponds not to a preorder internal
  to sheaves, but to a stack of preorders.} and so we may regard order
theory internal to a topos as another kind of enriched category
theory.  Both the monoidal poset $(\cat 2, \wedge, \top)$ (seen as a
one-object bicategory) and Walters' bases for enrichment are instances
of a general class of well-behaved bicategories, the
\emph{quantaloids}~\cite{Rosenthal1991Free}, whose enriched category
theory~\cite{Stubbe2005Categorical,Stubbe2006Categorical} behaves like
non-standard order theory. A quantaloid is a bicategory whose homs are
complete posets, and whose composition preserves joins in each
variable. The enrichment bases which give rise to faithful functors
over a base $\B$ are also quantaloids, and this justifies our
viewing their theory as non-standard order theory: our main result
can then be seen as saying that, for this non-standard order
theory, the complete preorders are the topological functors into $\B$.

Beyond proving our basic equivalence, we give an application to
\emph{MacNeille completions}. The classical MacNeille
completion~\cite{MacNeille1937Partially} of a poset $P$ is the
smallest complete poset into which $P$ embeds, and may be constructed
by Dedekind cuts: its elements are pairs $(L,U)$ of subsets of $P$
wherein $L$ is the set of all lower bounds of $U$, and $U$ the set of
all upper bounds of $L$. In~\cite{Herrlich1976Initial} is described a
more general ``MacNeille completion'', which turns a faithful functor
into $\B$ to a topological one, and includes the classical MacNeille
completion as the special case $\B = 1$. We unify these constructions by showing
that they are an instance of the general notion of ``MacNeille
completion'' for enriched categories first described by
Isbell~\cite{Isbell1960Adequate}.

\section{Topological functors}
Since questions of size are relevant in what follows, let us first
make clear our conventions. A set will be called \emph{small} if it
lies within some fixed Grothendieck universe $\kappa$, and
\emph{large} otherwise. All categories in this paper will be assumed
to have small hom-sets and a (possibly large) set of
objects\footnote{Our conventions deviate here from those commonly used
  in the categorical topology literature, where a small category is
  defined as one with a set of objects, and a large category may have
  a proper class of them. This avoids problems such as the failure of
  the collection of subclasses of a fixed class to form a class.}.

We begin by recalling the definition of topological
functor~\cite{Antoine1966Extension,Roberts1968A-characterization}. 

\begin{Defn}\label{defn:topo}
  Let $p \colon \E \to \B$ be a faithful functor, and $I$ a (possibly
  large) set. Given objects $(x_i \in \E)_{i \in I}$ and
  morphisms $(g_i \colon px_i \to z \in \B)_{i \in I}$, a \emph{final
    lifting} is an object $\bar z \in \E$ with $p \bar z = z$ such
  that, for any
   $\theta \in \B(z, pe)$,
     \begin{equation}\tag{*}\label{eq:lift}
    \text{$z \xrightarrow{\theta} pe$ lifts to a map $\bar z \to e$ iff each
      $px_i \xrightarrow{\theta . g_i} pe$
      lifts to a map $x_i \to e$.}
  \end{equation}    
The functor $p$ is called \emph{topological} if it admits all final
liftings.
\end{Defn}

\begin{Exs}
  
  The basic example of a topological functor is the forgetful functor
  $\cat{Top} \to \cat{Set}$: given spaces $(X_i \mid i \in I)$ and
  functions $(g_i \colon UX_i \to Z)$, we obtain a final lifting by
  equipping $Z$ with the topology in which $V \in \O(Z)$ just when
  $g_i^{-1}(V) \in \O(X_i)$ for each $i$.  Similarly, the categories
  of quasitopological spaces, limit spaces, filter spaces,
  subsequential spaces and so on, all admit topological forgetful
  functors to $\cat{Set}$; see~\cite{Dubuc1979Concrete}. Other
  interesting examples of categories topological over $\cat{Set}$
  include the category of bornological spaces; the category $\F$ of
  sets equipped with a filter of subsets~\cite{Blass1977Two-closed};
  and the categories of diffeological and Chen
  spaces~\cite{Baez2011Convenient}.

  An easy way of obtaining topological functors whose codomain is not
  $\cat{Set}$ is using the result that, if $\mathbb T$ is a small
  category with finite limits, and $p \colon \E \to \B$ is
  topological, then so too is $\cat{Lex}(\mathbb T, \E) \to
  \cat{Lex}(\mathbb T, \B)$. So, for example, the forgetful functors
  from topological groups or topological vector spaces to groups or
  vector spaces are topological.  Another class of examples to bear in
  mind are those with $\B = 1$; as in the introduction, a faithful
  functor into $1$ is just a (large) preorder, and such a functor is
  topological just when the preorder admits all joins.
\end{Exs}

\begin{Rk}\label{rk:topoduality}
  Dually, an \emph{initial lifting} of a family $(g_i \colon z \to
  px_i \in \B)_{i \in I}$ is a final lifting with respect to $p^\op$;
  and clearly, a functor $p$ admits all initial liftings just when
  $p^\op$ is topological. One might be tempted to call such a $p$
  \emph{optopological}, but it turns out that this is unnecessary: a
  functor is topological if and only if its opposite is so. This is the
  \emph{topological duality
    theorem}~\cite{Antoine1966Extension,Roberts1968A-characterization};
  in the case $\B = 1$, it reduces to the result that a preorder
  admits all meets if and only if it admits all joins. We return
  to this point in Section~\ref{sec:duality} below.
\end{Rk}


\begin{Rk}
  The definition of a topological functor is sometimes taken to
  include extra side-conditions. One is that it should be
  \emph{amnestic}---meaning that the fibres are posets, not
  preorders. This requirement is inessential for the basic
  theory. Another common side-condition is that the fibres should be
  small categories; again, this is unnecessary for the basic theory.
\end{Rk}

\section{Faithful functors as enriched categories}
Towards our characterisation of topological functors, we now describe
how faithful functors into a fixed base category $\B$ may be seen as
categories enriched in an associated quantaloid.  As in the
introduction, a \emph{quantaloid}~\cite{Rosenthal1991Free} is a
bicategory $\Q$ whose hom-categories $\Q(A,B)$ are complete (small)
lattices, and whose whiskering functors $(\thg) \circ f \colon \Q(B,
C) \to \Q(A,C)$ and $g \circ (\thg) \colon \Q(A,B) \to \Q(A,C)$
preserve all joins.  It follows that these whiskering functors have
right adjoints, denoted by
\begin{equation*}
  [f, \thg] \colon \Q(A,C) \to \Q(B,C) \qquad \text{and} \qquad \{g,
  \thg\} \colon \Q(A,C) \to \Q(A,B)
\end{equation*}
respectively. Existence of these right adjoints is the
bicategorical counterpart to a monoidal category's being left and
right closed, and ensures that there is a workable theory of enriched
categories over such a base.

\begin{Exs}
  A one-object quantaloid is a (unital) \emph{quantale} $(Q, \&, 1)$
  in the sense of~\cite{Mulvey1986}: a complete lattice $Q$ with an
  associative unital multiplication $\& \colon Q \times Q \to Q$ that
  preserves joins in each variable. In particular, any locale $L$
  yields a quantale $(L, \wedge, \top)$. Another important example,
  due to Lawvere~\cite{Lawvere1973Metric}, is the quantale $(\mathbb
  R_+, +, 0)$ of non-negative real numbers with the reverse of the
  usual ordering.

  Quantaloids with more than one object arise, amongst other places,
  in the work of Walters~\cite{Walters1981Sheaves,Walters1982Sheaves}:
  given a topological space (or locale) $X$, he considers the
  quantaloid $\Q_X$ whose objects are open sets of $X$, and for which
  $\Q_X(U,V)$ is the complete lattice of open sets of $X$ contained in
  $U \cap V$. Analogously, if $(\C, J)$ is a (standard) site, then there is a
  quantaloid $\Q_{\C}$ whose objects are those of $\C$, and for which
  $\Q_\C(X,Y)$ is the complete lattice of subsheaves of $\C(\thg, X
  \times Y)$. See~\cite{Heymans2012Elementary} for an abstract
  characterisation of such quantaloids.
\end{Exs}


The notions of category, functor and transformation enriched in a
quantaloid are particular cases of the bicategorical ones
of~\cite{Street1983Enriched}, though rather easier to state in this
special case due to the partially ordered homs, which ensure that all
$2$-cell axioms are automatically satisfied. For more on
quantaloid-enriched category theory,
see~\cite{Stubbe2005Categorical,Stubbe2006Categorical}.
\begin{Defn}
If $\Q$ is a quantaloid, a \emph{$\Q$-enriched category}, or
\emph{$\Q$-category} $\C$, is given by:
\begin{itemize}
\item A (possibly large) set of objects $\ob \C$;
\item For each $x \in \ob \C$, an
  object $\abs x \in \Q$, called the \emph{extent} of $x$; and
\item For all $x, y \in \ob \C$, a hom-object $\C(x,y) \in \Q(\abs x, 
  \abs y)$;
\end{itemize}
all such that
\begin{itemize}
\item For all $x \in \ob \C$, we have $1_{\abs x} \leqslant \C(x,x)
  $ in $\Q(\abs x, \abs x)$;
\item For all $x,y,z \in \ob \C$, we have $\C(y,z) \circ \C(x,y)
  \leqslant \C(x,z)$ in $\Q(\abs x, \abs z)$.
\end{itemize}
If $\C$ and $\D$ are $\Q$-categories, a \emph{$\Q$-functor} $F \colon
\C \to \D$ is an extent-preserving function $F \colon \ob \C \to \ob
\D$ such that $\C(x,y) \leqslant \D(Fx,Fy)$ in $\Q(\abs x, \abs y)$
for all $x,y \in \ob \C$.  Between $\Q$-functors $F,G \colon \C \to
\D$ there exists at most one $\Q$-transformation, which exists just when $1_{\abs x} \leqslant \D(Fx, Gx)$
for all $x \in \ob \C$, and will then be notated as $F \leqslant G$. We write $\Q\text-\cat{CAT}$ for the locally
posetal $2$-category of (possibly large) $\Q$-categories.
\end{Defn}
\begin{Ex}
  When $\Q$ is a quantale seen as a one-object quantaloid,
  $\Q$-enriched categories are categories enriched in the quantale
  seen as a monoidal poset. For example, categories enriched in the
  one-object quantaloid $(\cat 2, \wedge, \top)$ are are preordered
  sets; whilst Lawvere showed in~\cite{Lawvere1973Metric} that
  categories enriched in the one-object quantaloid $(\mathbb R_+, +,
  0)$ are (generalised) metric spaces. For the quantaloid $\Q_X$
  associated to a topological space $X$, Cauchy-complete skeletal $\Q_X$-enriched
  categories are, as in the introduction, posets internal to
  $\cat{Sh}(X)$, and correspondingly for the quantaloid associated to
  a site $(\C, J)$.
\end{Ex}

We now describe the quantaloid-enrichments that will be relevant in
this paper.
 \begin{Defn}\cite{Rosenthal1991Free}
   The \emph{free quantaloid} $\Q_\B$ on an ordinary category $\B$ has
   the same objects as $\B$; morphisms $U \colon X \tor Y$ in $\Q_\B$
   are subsets $U \subseteq \B(X,Y)$, ordered by inclusion; the
   composition of $U \colon X \tor Y$ with $V \colon Y \tor Z$ is
   given by $V \circ U = \{v \circ u \mid v \in V, u \in U\} \subseteq
   \B(X,Z)$; while the identity map at $X$ is $\{1_X\} \colon X \tor
   X$.  The right adjoints to whiskering are constructed as follows
   for each $U \colon X \tor Y$, $V \colon Y \tor Z$ and $W \colon X
   \tor Z$ in $\Q_\B$:
\begin{align*}
  [U, W] &= \{ v \in \B(Y,Z) \mid vu \in W \text{ for all } u \in U \}\\
 \text{and} \quad
  \{V, W\} &= \{ u \in \B(X,Y) \mid vu \in W \text{ for all } v \in V
  \}\rlap{ .}
\end{align*}
 \end{Defn}



 Given an ordinary category $\B$, consider now what it is to give a
 category enriched in the free quantaloid $\Q_\B$. We have
 a set $\ob \E$ of objects; for each $x \in \ob \E$ an object $\abs x
 \in \ob \B$; and for each $x,y \in \ob \E$ a subset $\E(x,y)
 \subseteq \B(\abs x, \abs y)$, in such a way that $1_{\abs x} \in
 \E(x,x)$ and whenever $f \in \E(x,y)$ and $g \in \E(y,z)$ we have
 also $g \circ f \in \E(x,z)$. These data are clearly those of an
 ordinary category $\E$ together with a faithful functor $p \colon \E
 \to \B$ sending $x$ to $\abs x$. Similar calculations with respect to
 functors and transformations now
 show that:
\begin{Prop}\label{prop:faithful}
  If $\Q_\B$ is the free quantaloid on the ordinary category $\B$,
  then the $2$-category $\Q_\B\text-\cat{CAT}$ is $2$-equivalent to the
  full sub-$2$-category of the strict slice 
  $\cat{CAT}/\B$ on the faithful functors.
\end{Prop}

\section{Totality for quantaloid-enriched categories}
In this section, we discuss totality in the context of
quantaloid-enriched categories. The notion of \emph{totality} of a
category was introduced
in~\cite{Street1974Elementary,Street1978Yoneda} in an abstract context
broad enough to encompass ordinary categories but also enriched and
internal ones. An ordinary (possibly large) category $\C$ is called
\emph{total} if its Yoneda embedding $\C \to [\C^\op, \cat{Set}]$
admits a left adjoint\footnote{Our standing assumptions, that every
  category be locally small, ensure that the Yoneda embedding $\C \to
  [\C^\op, \cat{Set}]$ exists; yet if $\C$ is not small, then
  $[\C^\op, \cat{Set}]$ may itself fail to satisfy this same standing
  assumption. This is, in fact, an instance of the delicacy described in
  the following paragraph: the category of presheaves on the large
  $\cat{Set}$-category $\C$ need not form a $\cat{Set}$-category.}.
Totality implies cocompleteness, the existence of all small colimits,
but also the existence of certain large ones, which can be used, for
example, to avoid solution-set conditions in the adjoint functor
theorem;
see~\cite[Theorem~18]{Street1974Elementary}. 
Most categories arising in mathematical practice are total: for
instance, any locally presentable category, in particular, any
Grothendieck topos; any category monadic over $\cat{Set}$; any
category admitting a topological functor to $\cat{Set}$; and so on.

As explained in~\cite{Kelly1986A-survey}, defining totality in the
enriched context requires some delicacy: the na\"ive
definition---involving a left adjoint to the Yoneda embedding into the
enriched presheaf category---is complicated by the fact that the
presheaves on a large $\W$-category in general only form a
$\W'$-category for some universe-enlargement $\W'$ of $\W$. One can
avoid this issue by restating the definition of totality purely in
terms of the large colimits required to exist, but when enriching in a
quantaloid $\Q$ this is unneccesary. Because the hom-categories of
$\Q$ are complete \emph{posets}, they admit all of the large limits
necessary for the presheaves on a large $\Q$-category to exist as a
$\Q$-category; and so the na\"ive definition of totality is, in this
context, valid.

\begin{Defn}\label{def:presheaves}
Let $\Q$ be a quantaloid, and $\C$ a $\Q$-enriched category. A
\emph{presheaf} $\phi$ on $\C$ is given by:
\begin{itemize}
\item An object $\abs \phi \in \Q$, the \emph{extent} of $\phi$; and
\item For each $x \in \ob \C$, an arrow $\phi(x) \colon \abs x \to \abs
  \phi$ in $\Q$,
\end{itemize}
such that we have $\phi(y)
\circ \C(x,y) \leqslant \phi(x) \colon \abs x \to \abs \phi$ for all $x, y \in \ob \C$. The
\emph{presheaf $\Q$-category} $\P \C$ has these presheaves as objects,
with the specified extents, and
hom-arrows given
by 
\[
\P\C(\phi, \psi) = \textstyle\bigwedge_{x \in \C} [\phi(x), \psi(x)]
\colon \abs \phi \to \abs \psi\rlap{ .}
\]
The \emph{Yoneda embedding} $Y \colon \C \to \P\C$ sends $x$ to the
representable presheaf $\C(\thg, x)$ with extent $\abs x$ and with
components $\C(\thg, x)(y) = \C(y,x)$. 
\end{Defn}
\begin{Defn}
A $\Q$-category  $\C$ is called \emph{total}
if the Yoneda embedding $Y \colon \C \to \P \C
$ admits a left adjoint in $\Q\text-\cat{CAT}$.
\end{Defn}
We prefer to use \emph{total} rather than \emph{cocomplete} for this
notion, reserving the latter to mean, as is standard, ``having all
small colimits''. Note that in~\cite{Stubbe2005Categorical}, Stubbe
uses \emph{cocomplete} for what appears to be our \emph{total}; but he
in fact works under the restriction that $\Q$ should be a \emph{small}
quantaloid, and $\C$ a \emph{small} $\Q$-category, so that his
nomenclature is compatible with ours.
\begin{Ex}\label{ex:pctotal}
  For any $\Q$-category $\C$, $\P \C$ is total; the left adjoint
  $\mu \colon \P\P\C \to \P\C$ of the Yoneda embedding is defined at
  $\Phi$ by $\mu(\Phi)(x) = \bigvee_{\phi \in \P \C} \Phi(\phi) \circ \phi(x)$.
\end{Ex}
As explained above, totality is equivalent to the existence of certain large
colimits.  In the quantaloid-enriched case, it is actually equivalent 
to the existence of \emph{all} large colimits.
\begin{Defn}
  Let $\Q$ be a quantaloid and $F \colon \I \to \C$ a $\Q$-functor.
  \begin{itemize}
  \item The \emph{singular $\Q$-functor} $\C(F,\thg) \colon \C \to \P \I$
    sends $c \in \C$ to $\C(F, c)$ with extent $\abs c$ and
    components $\C(F,c)(x) = \C(Fx,d) \colon \abs x \to \abs c$.
\vskip0.5\baselineskip
  \item Given $\phi \in \P \I$, a \emph{weighted colimit of $F$ by
    $\phi$} is a left adjoint  to $\C(F,\thg)$ at $\phi$,
     given by an object $\phi \star F$ of $\C$ with extent $\abs
    \phi$ such that
    $\C(\phi \star F, c) = \P\I(\phi, \C(F, c))$ in $\Q(\abs \phi,
    \abs c)$ for all $c \in \C$. 
  \end{itemize}
\end{Defn}
\begin{Prop}\label{prop:totalallcolims}\cite[Corollary 5.4]{Stubbe2005Categorical}
  A $\Q$-category $\C$ is total if and only if it admits all (possibly
  large) colimits.
\end{Prop}
\begin{proof}
  All colimits exist in $\C$ just when every singular functor $\C(F,
  \thg) \colon \C \to \P\I$ admits a left adjoint. Now $\C(F, \thg)$
  is the composite of $Y \colon \C \to \P\C$ with the $\Q$-functor
  $F^\ast \colon \P \C \to \P \I$ defined by $(F^\ast \phi)(x) =
  \phi(Fx)$, and $F^\ast$ always has a left adjoint $F_!  \colon \P\I
  \to \P\C$, defined by $(F_!  \psi)(c) = \bigvee_{x \in \I} \psi(x)
  \circ \C(c,Fx)$; whence every $\C(F, \thg)$ will have a left adjoint
  just when $Y$ does, that is, just when $\C$ is total.
\end{proof}

\section{Topological = total}
\label{sec:topological-=-total}
We are now ready to prove our main result: that total categories
enriched in the free quantaloid $\Q_\B$ correspond to topological
functors into $\B$.
Given a faithful functor $p \colon \E \to \B$, corresponding via
Proposition~\ref{prop:faithful} to a $\Q_\B$-category $\bar \E$, we
will notate by
\begin{equation}\label{eq:freecocompl}
  \cd[@!C]{
    \E \ar[rr]^Y \ar[dr]_{p} & & \P \E \ar[dl]^{\P p} \\ &
    \B
  }
\end{equation}
the arrow of $\cat{CAT} / \B$ corresponding to the Yoneda $\Q_\B$-functor
$\bar \E \to \P \bar \E$. Unfolding the definitions, an
object of $\P \E$ over $z \in \B$ is an object of $\P \bar \E$
with extent $z$; thus given by 
subsets $\phi(x) \subseteq \B(px,z)$ for each $x \in \E$ such that
\[
\text{$f \in \phi(y)$ and $g \in \E(x,y)$} \quad \Longrightarrow \quad
f \circ p(g) \in \phi(x)\rlap{ ;}
\]
in other words, by a subfunctor $\phi \subseteq \B(p\thg, z)$, which
we call a \emph{$p$-sieve on $z$}. Given another $p$-sieve $\psi
\subseteq \B(p \thg, w)$, a morphism $\phi \to \psi$ in $\P \E$ is a
map $\theta \colon z \to w$ in $\B$ such that $\theta \circ f \in
\psi(x)$ whenever $f \in \phi(x)$. The functor $Y \colon \E \to \P \E$
sends $e \in \E$ to the $p$-sieve $\E(\thg, e) \subseteq \B(p\thg,
pe)$. Note that $\P p$ is topological by Example~\ref{ex:pctotal}.

\begin{Rk}
  The construction of $\P p$ from $p$ is well-known in the categorical
  topology literature; in the notation of~\cite{Herrlich1976Initial},
  it is the completion $(\E^{-2}, p^{-2})$ of $(\E, p)$.
\end{Rk}

We now give our main result; in (ii), we allow ourselves to identify a
$p$-sieve $\phi \subseteq \B(p\thg, z)$ with the family of maps $(g \colon
px \to z)_{x \in \E, g \in \phi(x)}$.

\begin{Thm}\label{thm:main}
  Let $p \colon \E \to \B$ be a faithful functor and $\bar \E$ the
  corresponding $\Q_\B$-category. The following are equivalent:
\begin{enumerate}[(i)]
\item $p$ is topological;
\item $p$ admits all final liftings of $p$-sieves;
\item $Y \colon \E \to \P \E$ in~\eqref{eq:freecocompl} admits a left adjoint over $\B$;
\item $\bar \E$ is total.
\end{enumerate}
\end{Thm}

\begin{proof}
  Clearly (i) implies (ii); conversely, suppose that $p$ admits final
  liftings of $p$-sieves; we will show that any family $\textbf g =
  (g_i \colon px_i \to z)_{i \in I}$ admits a final lifting. Given
  such a $\mathbf g$, form the $p$-sieve $\phi \subseteq \B(p\thg, z)$
  with
  \begin{equation}\label{eq:generatedsieve}\phi(x) = \{f \colon px \to z \mid f = g_i \circ pk
  \text{ for some $i \in I$ and $k \colon x \to x_i$ in $\E$}\}\rlap{
    ,}\end{equation} and let $\bar z$ be a final lifting of $\phi$. We
  claim that $\bar z$ is also a final lifting of $\mathbf g$; which
  will follow so long as for all $\theta \in \B(z, pe)$,
     \begin{align*}
    & \text{$\theta \circ g_i \colon px_i \to pe$
      lifts to a map $x_i \to e$ for all $i \in I$}\\
    \Longleftrightarrow \quad
    & \text{$\theta \circ g \colon px \to pe$
      lifts to a map $x \to e$ for all $g \in \phi(x)$}\rlap{ .}
  \end{align*}   
  The leftward implication follows since each $g_i \in \phi(x_i)$; the
  rightward since each $g \in \phi(x)$ factors as $g_i \circ pk$. Thus
  (ii) implies (i).
  We now show that (ii) $\Leftrightarrow$ (iii).  To say that $Y$ has
  a left adjoint at an object $\phi \subseteq \B(p\thg, z)$ of $\P \E$
  is to say that there is an object $\bar \phi \in \E$ with $p\bar
  \phi = z$ such that, for all $\theta \in \B(z, pe)$,
  \begin{equation}\tag{\textdagger}\label{eq:topoproof}
    \text{$ z \xrightarrow{\theta} pe$ lifts to a map $\bar \phi \to e$
      in $\E$
      iff it lifts to a map $\phi \to Ye$ in $\P\E$.}
  \end{equation}   
  But to say that $\theta$ lifts to a map $\phi \to Ye$ is to say that
  $\theta \circ g \colon px \to pe$ lifts to a map $x \to e$ for every
  $g \in \phi(x)$, and so condition (\textdagger) says precisely that
  $\bar \phi$ is a final lifting of the sieve $\phi$.
  Finally, (iii) $\Leftrightarrow$ (iv) by 
  Proposition~\ref{prop:faithful}.
\end{proof}

\section{Duality}\label{sec:duality}
In Remark~\ref{rk:topoduality} we mentioned the \emph{topological
  duality theorem}, which states that a functor $p \colon \E \to \B$
is topological if and only if $p^\op \colon \E^\op \to \B^\op$ is
so. In this section, we explain this result in terms of a general result of
quantaloid-enriched category theory: namely, that the notion of total
$\Q$-category is self-dual.
The starting point is the adjoint functor theorem for $\Q$-categories.

\begin{Prop}\label{prop:preadj}\cite[Proposition
  4.6]{Stubbe2005Categorical}
  If $F \colon \C \to \D$ is a $\Q$-functor and $d \in \D$, then a
  right adjoint for $F$ at $d$ is given by $Gd = \D(F, d) \star 1_\C$
  whenever this colimit exists in $\C$ and is preserved by $F$.
\end{Prop}
\begin{proof}
  Under the stated hypotheses, we must verify that $\C(c, Gd) = \D(Fc,
  d)$. On the one hand, we have $1_{\abs d} \leqslant \C(Gd, Gd) = \P\C(\D(F,
  d), \C(1, Gd))$, whence $\D(Fc,d) \leqslant \C(c, Gd)$ as required;
  conversely, since $FGd$ has the universal property of $\D(F,d) \star
  F$, we have $1_{\abs d} \leqslant \P\C(\D(F,d),\D(F,d)) = \D(FGd,
  d)$ whence $\C(c,Gd) \leqslant \D(FGd, d) \circ \D(Fc,FGd) \leqslant
  \D(Fc, d)$ as required.
\end{proof}

\begin{Cor}[Adjoint functor theorem]\label{cor:adjoint}
  If $\C$ is a total $\Q$-category, and $F \colon \C \to \D$ preserves
  all (large) colimits, then $F$ has a right adjoint in $\Q$-$\cat{CAT}$.
\end{Cor}

We now apply this result to show that the notion of totality for
$\Q$-categories is self-dual. First we make clear the sense of that
duality.

  \begin{Defn}
    The \emph{copresheaf category} $\P^\dagger \C$ on a $\Q$-category
    $\C$ has as objects $\phi$, families of maps $\phi(x) \colon
    \abs \phi \to \abs x$ satisfying $\C(x,y) \circ \phi(x) \leqslant
    \phi(y)$, and hom-objects defined by $\P^\dagger\C(\phi, \psi) =
    \textstyle\bigwedge_{x \in C} \{\psi(x), \phi(x)\}$ (note the
    reversal of order!). The dual Yoneda embedding $Y^\dagger \colon
    \C \to \P^\dagger \C$ sends $x$ to $\C(x, \thg)$.
  We say that $\C$ is \emph{cototal} if $Y^\dagger \colon \C \to \P^\dagger \C$
  admits a right adjoint.
\end{Defn}
\begin{Rk}
  We can go on to define a \emph{weighted limit} $\{\psi, F\}$ of a
  $\Q$-functor $F \colon \I \to \C$ by a weight $\psi \in \P^\dagger
  \I$ as a right adjoint at $\psi$ to the dual singular functor $\C(1,
  F) \colon \C \to \P^\dagger \I$; then as in
  Proposition~\ref{prop:totalallcolims}, we may conclude that a $\Q$-category
  is cototal just when it admits all large limits.
\end{Rk}
\begin{Rk}\label{rk:opposite}
 Any $\Q$-category $\C$ has an opposite $\C^\op$ which is a
 $\Q^\op$-category; now $\P^\dagger$ as defined above is equally
 $(\P(\C^\op))^\op$, and $Y^\dagger$ the opposite of $Y \colon \C^\op
 \to \P(\C^\op)$. In these terms, the $\Q$-category $\C$ is cototal
 just when $\C^\op$ is a total $\Q^\op$-category.
\end{Rk}
For a faithful functor $p \colon \E \to \B$, seen as a category
enriched over the free quantaloid $\Q_\B$, applying the copresheaf
construction yields the topological functor $\P^\dagger p \colon
\P^\dagger \E \to \B$ for which objects over $z$ are cosieves $\phi
\subseteq \B(z, p\thg)$ and morphisms from $\phi$ to $\psi \subseteq
\B(w, p\thg)$ are morphisms $\theta \colon z \to w$ of $\B$ such that
$f \in \psi(x)$ implies $f \circ \theta \in \phi(x)$;
in the notation of~\cite{Herrlich1976Initial}, this is the completion
$(\E^{2}, p^{2})$ of $(\E, p)$.

\begin{Thm}\label{thm:duality}\cite[Proposition 5.10]{Stubbe2005Categorical}
  A $\Q$-category $\C$ is total just when it is cototal.
\end{Thm}
\begin{proof}
  A straightforward calculation shows that the dual Yoneda embedding
  $Y^\dagger \colon \C \to \P^\dagger \C$ preserves all colimits; thus, if $\C$ is total, then
  $Y^\dagger$ has a right adjoint, whence $\C$ is cototal. The
  converse implication is dual.
\end{proof}
\begin{Rk}\label{rk:limcolim}
  More generally, Proposition~\ref{prop:preadj} and its dual provide
  formulas for computing limits in $\Q$-categories in terms of
  colimits and vice versa; c.f.~\cite[Proposition
  5.8]{Stubbe2005Categorical}. The situation is encapsulated in terms of
  the~\emph{Isbell adjunction}~\cite{Isbell1960Adequate}
  \begin{equation}\label{eq:isbell}
  \cd[@C+3em]{
    \P\C \ar@<4.5pt>[r]^{\mathord \uparrow = \P\C(1, Y)} \ar@{}[r]|{\bot} &
    \P^\dagger \C \ar@<4.5pt>[l]^{\mathord \downarrow =
      \P^\dagger\C(Y^\dagger, 1)}
  }
  \end{equation}
  between the singular functor of $Y^\dagger$ and the dual singular
  functor of $Y$. Given $\psi \in \P^\dagger \C$, the weighted limit
  $\{\psi, 1_\C\}$ can be computed as $(\mathord \downarrow \psi)
  \star 1_\C$ whenever the latter colimit exists. Dually, $\phi \star
  1_\C$ can be computed as the limit $\{\mathord \uparrow \phi,
  1_\C\}$. 
\end{Rk}

In the context of enrichment over a free quantaloid $\Q_\B$, taking
Theorem~\ref{thm:duality} together with Remark~\ref{rk:opposite} and
Theorem~\ref{thm:main} recovers the topological duality theorem of
Remark~\ref{rk:topoduality}. From Remark~\ref{rk:limcolim}, we obtain
the explicit formula for computing final liftings along $p$ in terms
of initial ones: the Isbell adjunction for $p \colon \E \to B$ sends a
sieve $\phi \subseteq \B(p\thg, z)$ to the cosieve $\mathord \uparrow
\phi$, and a cosieve $\psi \subseteq \B(z, p
\thg)$ to the sieve $\mathord \downarrow \psi$ defined by the
following formulae:
\begin{equation}\label{eq:isbelltopo}
\begin{gathered}
\mathord \uparrow \phi = \{g \colon z \to px \mid g \circ h \colon
py \to px \text{ lifts to $y \to x$ $\forall h \colon py \to z$
  in $\phi$}\}
\\
\mathord \downarrow \psi = \{h \colon py \to z \mid g \circ h \colon
py \to px \text{ lifts to $y \to x$ $\forall g \colon z \to px$
  in $\psi$}\}\rlap{ .}
\end{gathered}
\end{equation}
Thus given a family of maps $\mathbf g =
(g_i \colon px_i \to z)_{i \in I}$, on forming the sieve $\phi
\subseteq \B(p\thg, z)$ they generate, and the conjugate cosieve
$\mathord \uparrow \phi \subseteq \B(p\thg, z)$,
Remark~\ref{rk:limcolim} asserts that an initial lifting for $\mathbf
g$ can be obtained as a final lifting for $\mathord \uparrow \phi$;
and dually.

\section{MacNeille completions}
In this final section, we give an application of our main result to
MacNeille completions. As in the introduction, the \emph{MacNeille
  completion} of a preorder $X$ is the poset $\R X$ whose elements are
pairs $(L, U)$ of subsets of $X$ with $L = {\downarrow\! U}$ and $U =
{\uparrow\! L}$, ordered by $(L, U) \leqslant (L', U')$ iff $L
\subseteq L'$ (equivalently, $U \supseteq U'$); here we write
\begin{equation*}
\downarrow U  = \{\ell \in X \mid \ell \leqslant u \text{ $\forall u \in U$}\}
\quad \text{and} \quad
\uparrow L  = \{u \in X \mid \ell \leqslant u \text{  $\forall \ell \in
  L$}\}\rlap{ .}
\end{equation*}
$\R X$ is a complete lattice, with meets $\bigwedge_i (L_i, U_i) =
(\bigwedge_i L_i, \mathord \uparrow(\bigwedge_i L_i))$ and joins
defined dually, and there is an order-embedding $J \colon X \to \R X$
sending $x$ to $(\downarrow\!\{x\}, \uparrow\!\{x\})$ that preserves
all meets and joins that exist in $X$ and is join- and
meet-dense. These properties in fact serve to characterise $\R X$ up
to equivalence of preorders; it follows that $\R$ is idempotent, in
the sense that $J$ is an equivalence whenever $X$ is a complete
preorder. There is an alternate characterisation in terms of
cut-continuous maps~\cite{Bishop1978A-universal}. A \emph{lower cut}
in a poset is a subset of the form $\downarrow U$, and a map of
preorders $f \colon X \to Y$ is called \emph{lower cut-continuous}
when $f^{-1}$ preserves lower cuts. Now the assignation $X \mapsto \R
X$ provides a bireflection of the $2$-category of preorders and lower
cut-continuous maps to its full sub-$2$-category on the complete
preorders.

In~\cite{Herrlich1976Initial}, Herrlich introduced the notion of
\emph{MacNeille completion} of a faithful functor to a topological
one; this completion has the same good properties of the MacNeille
completion of a preorder, and indeed reduces to it in the case where
$\B = 1$. In this section, we will justify the nomenclature by
exhibiting the topological MacNeille completion as an instance of the
general process of ``MacNeille completion'' for categories enriched
over an arbitrary base; the construction here was first described (for
unenriched categories) by
Isbell~\cite{Isbell1960Adequate}. It is rather less well-behaved for
general enriched categories than for preorders, but when the
enrichment base happens to be a quantaloid, all of the good properties
of the order-theoretic case are retained; it is this situation that we
shall now describe.

We motivate the construction with some remarks about the MacNeille
completion of a preorder $X$. Note that if $(L,U) \in \R X$, then $L$
is a downset and $U$ an upset in $X$; and on viewing $X$ as a category
enriched in the one-object quantaloid $(\cat 2, \wedge, \top)$, such
downsets and upsets in $X$ are the respective objects of the presheaf
and copresheaf categories $\P X$ and $\P^\dagger X$. The operations
$\uparrow$ and $\downarrow$ described above are now precisely those of
the Isbell adjunction~\eqref{eq:isbell} for $X$. This adjunction is a
Galois connection between posets; and $\R X$ is one of the several
equivalent descriptions of the poset of fixpoints of this Galois
connection.

Motivated by this, we now describe the MacNeille completion of a general
quantaloid-enriched category. First a preparatory result on
adjunctions in the $\Q$-enriched context; the proof is entirely
straightforward and so omitted.
\begin{Prop}\label{prop:galois}
  For any adjunction $F \colon \C \rightleftarrows \D \colon G$ of
  $\Q$-categories, there are canonical isomorphisms induced by $F$ and
  $G$ between the
  following $\Q$-categories:
  \begin{enumerate}[(a)\quad\ \ ]
  \item The full replete image of $G$ in $\C$;
  \item The full subcategory of $\C$ on those objects $X$ with 
    $X \cong GFX$;
  \item The $\Q$-category of $GF$-algebras in $\C$;
  \item[(d)--(f)] The duals of (a)--(c);
  \item[(g)\quad\ \ ] The $\Q$-category whose objects are pairs $(c, d)$ with $Fc
    \cong d$ and $c \cong Gd$ and whose hom-object from $(c,d)$ to
    $(c',d')$ is $\C(c,c') = \D(d,d')$.
 \end{enumerate}
The $\Q$-categories (a)--(c) are reflective in $\C$, while
(d)--(f) are coreflective in $\D$.
\end{Prop}
We write $\cat{Fix}(F,G)$ for any of the isomorphic $\Q$-categories
(a)--(g).

\begin{Defn}
  The~\emph{MacNeille completion} $\R\C$ of a $\Q$-category $\C$ is
  the $\Q$-category $\cat{Fix}(\mathord \uparrow, \mathord
  \downarrow)$ associated to the Isbell adjunction~\eqref{eq:isbell};
  for concreteness, we take the representation in (b), so that $\R\C$
  comprises those $\phi \in \P\C$ with $\phi = \mathord \downarrow
  \mathord \uparrow \phi$.  Since $\uparrow$ and $\downarrow$ send
  representables to representables, the Yoneda embedding factors
  through $\R\C$ as $J \colon \C \to \R\C$, say.
\end{Defn}

\begin{Rk}\label{rk:limitsofreps}
  Since $\mathord \downarrow \colon \P^\dagger \C \to \P\C$ is right
  adjoint to the dual singular functor $\P\C(1, Y)$, it must send
  $\psi$ to $\{\psi, Y\}$. Thus every $\phi \in \R\C$ is a limit
  $\{\mathord \uparrow \phi, Y\}$ of representables in $\P\C$; on the
  other hand, as $\R\C$ is reflective in $\P\C$, it is
  limit-closed, and so $\R\C$ is in fact the closure of the
  representables in $\P\C$ under (large) limits.
\end{Rk}

When $\Q$ is the one-object quantaloid $(\cat 2, \wedge, \top)$, the
MacNeille completion of a $\Q$-category (=preorder) $X$ is the
classical MacNeille completion $\R X$. When $\Q$ is the one-object
quantaloid $(\mathbb R_+, +, 0)$, the MacNeille completion of
$\Q$-category (=generalised metric space) was identified
in~\cite{Willerton2013Tight} with its \emph{directed tight span}
completion. When $\Q$ is the quantaloid associated to a space $X$ or a
site $(\C, J)$, the MacNeille completion of a (Cauchy-complete,
skeletal) $\Q$-category is its MacNeille completion internal to the
topos of sheaves.  Finally, for the free quantaloids which are our
main concern in this paper, we have that:
\begin{Prop}
  Let $\Q_\B$ be the free quantaloid on a category $\B$. If $p \colon
  \E \to \B$ is a faithful functor corresponding to a $\Q_\B$-category
  $\bar \E$, then the MacNeille completion of $p$ \emph{qua} faithful
  functor corresponds to the MacNeille completion of $\bar \E$
  \emph{qua} quantaloid-enriched category.
\end{Prop}
\begin{proof}
  The faithful functor $\R p \colon \R\E \to \B$ corresponding to $\R
  \overline \E$ has as domain the full subcategory of $\P\E$ on those
  sieves $\phi \subseteq \P(p\thg, b)$ with $\phi = \mathord
  \downarrow \mathord \uparrow \phi$, for $\mathord \downarrow$ and
  $\mathord \uparrow$ defined as in~\eqref{eq:isbelltopo}.  This is
  precisely the construction of the MacNeille completion of $p$ given in in~\cite{Herrlich1976Initial}.
\end{proof}

Our remaining results give characterisations of the MacNeille
completion of a $\Q$-category that extend those for the classical
case; having identified the scope of the notions in particular
examples, we shall henceforth feel free to work only in the general
situation. We begin with a preparatory result concerning density.
\begin{Prop}
  For a $\Q$-functor $F \colon \C \to \D$, the following
  are equivalent:
  \begin{enumerate}[(a)]
  \item $\D(F, 1) \colon \D \to \P\C$ is fully faithful;
  \item Each $d \in \D$ is the colimit $\D(F, d) \star F$;
  \item Each $d \in \D$ is a colimit $\phi \star F$ for some $\phi \in \P\C$.
  \end{enumerate}
\end{Prop}
We call $F \colon \C \to \D$ with any of these equivalent properties
\emph{dense}. The equivalence of (a) and (b) is standard enriched
category theory~\cite[Theorem 5.1]{Kelly1982Basic}; that of (b) and
(c) is peculiar to the quantaloid-enriched case.
\begin{proof}
  (a) and (b) both say that $\D(d,d') = \P\C(\D(F,d), \D(F,d'))$ for
  all $d, d' \in \D$, and (b) clearly implies (c). It remains to show
  (c) $\Rightarrow$ (b); thus fixing $d \in \D$, we must show that
  $\P\C(\D(F,d), \D(F,d')) \leqslant \D(d,d')$ for all $d' \in \D$ (the converse
  inequality is automatic by functoriality of $\D(F,1)$). By
  assumption, $d$ is $\phi \star F$ for some $\phi$; and so $1_{\abs d}
 \leqslant \D(d,d) = \P\C(\phi, \D(F,d))$, whence
 $\P\C(\B(F,d), \B(F,d')) \leqslant \P\C(\phi, \B(F,d')) = \D(d,d')$
 as required.
\end{proof}

\begin{Prop}
  The MacNeille completion $J \colon \C \to \R\C$ 
  has the properties that:
\begin{enumerate}[(a)]
\item $\R\C$ is total;
\item $J$ is fully faithful;
\item $J$ is dense and codense.
\end{enumerate}
These properties 
imply that:
\begin{enumerate}[(a)]
\addtocounter{enumi}{3}
\item $J$ preserves all limits and colimits that exist in $\C$;
\item Any full embedding of $\C$ into a total $\Q$-category $\D$
  extends to a full embedding of $\R\C$ into $\D$.
\end{enumerate}
and moreover characterise $\R\C$ and $J$ up to equivalence; in
particular, $\R$ is idempotent in the sense that $J$ is an equivalence
whenever $\C$ is total.
\end{Prop}

\begin{proof}
  $\R\C$ is reflective in the total $\P\C$, and hence itself total;
  while $J$ is a factorisation of the fully faithful $Y \colon \C \to
  \P\C$ and so itself fully faithful. Codensity follows from
  Remark~\ref{rk:limitsofreps} and density from its dual.  For (d), a
  fully faithful dense functor $F \colon \A \to \B$ must preserve
  limits, since it factorises the limit-preserving Yoneda embedding;
  dually, codensity implies colimit-preservation. For (e), given $F
  \colon \C \to \D$ a full embedding with $\D$ total, let $\D'$ be the
  closure of $\C$ in $\D$ under colimits; then $\D'$ is again total,
  and $\C \to \D'$ is dense, so that $\D'$ can be identified with a
  limit-closed subcategory of $\P\C$ containing the
  representables. Now let $\D''$ be the closure of the representables
  in $\D'$ under limits; then by construction and
  Remark~\ref{rk:limitsofreps}, $\D'' \simeq \R\C$, as required.
  Finally, for the uniqueness, observe that in the construction just
  given, we have $\D' = \D$ if $J$ is dense, and then $\D = \D' =
  \D''$ if in addition $J$ is codense. The final clause follows since,
  if $\C$ is total, then $1_\C \colon \C \to \C$ satisfies properties
  (a)--(c).
\end{proof}

We conclude by describing a universal property of the MacNeille
completion of a $\Q$-category which generalises a
universal property~\cite{Bishop1978A-universal} of the classical
case.  In the following definition, recall that for any $\Q$-functor
$F \colon \C \to \D$, the $\Q$-functor $F^\ast \colon \P\D \to \P\C$
is defined at $\phi$ by $(F^\ast \phi)(x) = \phi(Fx)$.

\begin{Defn}
  A $\Q$-functor $F \colon \C \to \D$ is \emph{cut-cocontinuous}
  if $F^\ast \colon \P \D \to \P \C$ maps $\R \D$ into $\R\C$.
\end{Defn}
\begin{Prop}\label{prop:totalfunctorchar}
  $F \colon \C \to \D$ is cut-cocontinuous if and only if the singular
  functor $\D(F, 1) \colon \D \to \P\C$ lands inside $\R\C$.
\end{Prop}
\begin{proof}
  We must show that $F^\ast$ maps $\R\D$ into $\R\C$ if and only if it
  maps the representables into $\R\C$. For the non-trivial direction,
  note that $F^\ast$ preserves all limits, each $\phi \in \R\D$ is
  a limit of representables in $\P\D$, and $\R\C$ is closed under
  limits in $\P\C$.
\end{proof}
\begin{Prop}\label{prop:totalcolim}
  For any $\Q$-functor $F \colon \C \to \D$, we have
\begin{align*}
&\text{$F$ is a left adjoint} \\ \Longrightarrow \quad
&\text{$F$ is cut-cocontinuous} \\ \Longrightarrow \quad
&\text{$F$ preserves all (large) colimits.}
\end{align*}
If $\C$ is total, then all three conditions are equivalent.
\end{Prop}
\begin{proof}
  For the first implication, if $F$ has a right adjoint $G$, then
  $\D(F,1) = YG \colon \D \to \P\C$ clearly lands in $\R\C$. For the
  second implication, it suffices by
  Proposition~\ref{prop:totalallcolims} to show that a
  cut-cocontinuous $F$ preserves any colimit $v = \phi \star 1_\C$
  existing in $\C$. Now $\mathord \uparrow \phi = \P\C(\phi, Y) =
  \C(v, 1) = Y^\dagger v$; and since $\D(F, 1)
  \colon \D \to \P \C$ factors through $\R\C$, we have that
  $\P\C(\phi, \D(F, 1)) = \P^\dagger\C(\mathord \uparrow \phi,
  \mathord \uparrow \D(F,1)) = \P^\dagger\C(Y^\dagger v, \mathord
  \uparrow \D(F,1)) = \D(Fv, 1)$, so that $Fv$
  is $\phi \star F$ as claimed. The final clause follows from the
  adjoint functor theorem.
\end{proof}

Now let $\cat{CCOCTS}$ be the $2$-category of $\Q$-categories,
cut-cocontinuous $\Q$-functors, and $\Q$-transformations, and let
$\cat{TOT}$ denote the full sub-$2$-category on the total
$\Q$-categories; by the preceding proposition, the morphisms of
$\cat{TOT}$ are the functors preserving all colimits.
\begin{Prop}
  The MacNeille completion $J \colon \C \to \R\C$ is the value at
  $\C$ of a left biadjoint to the inclusion $2$-functor
   $\cat{TOT} \to \cat{CCOCTS}$.
\end{Prop}
\begin{proof}
  Let $\C$ be a $\Q$-category.  Since $\R\C$ is reflective in the
  total $\P\C$, it is itself total; moreover, $J \colon \C \to \R\C$
  is cut-cocontinuous by Proposition~\ref{prop:totalfunctorchar}, as
  its singular functor $\R\C(J, 1) \colon \R \C \to \P \C$ is the
  inclusion.  It remains
  to show that, for any total $\Q$-category $\D$, the restriction
  $\Q$-functor
  \begin{equation}\label{eq:kanres}
    \cat{CCOCTS}(\R\C, \D) \xrightarrow{(\thg) \circ J}
    \cat{CCOCTS}(\C, \D)
  \end{equation}
  is an equivalence of preorders; we do so by exhibiting an explicit
  pseudoinverse. Given a cut-cocontinous $F \colon \C \to \D$, define
  $F^{\#} \colon \R \C \to \D$ by $F^\#(\phi) = \phi \star F$. We must
  show that $F^\#$ is cut-cocontinuous; equivalently, by
  Proposition~\ref{prop:totalcolim}, that it is a left adjoint. By
  Proposition~\ref{prop:totalfunctorchar} and cut-cocontinuity of $F$,
  the singular functor $\D(F, 1) \colon \D \to \P\C$ factors through
  $\R\C$, as $H \colon \D \to \R\C$, say; then as $(\thg) \star F
  \colon \P\C \to \D$ is left adjoint to $\D(F,1)$, it follows that
  $F^\#$ is left adjoint to $H$, as required.  It remains to show that
  $(\thg)^\#$ is pseudoinverse to $(\thg) \circ J$, for which we need
  two things:
  \begin{enumerate}
  \item For any cut-cocontinuous $F \colon \C \to \D$, we have $F^\# J \cong F$; but
    for each $x \in \C$ we have $F^\#(Jx) = \C(\thg, x) \star F \cong
    Fx$ by the Yoneda lemma.
  \item For any (cut-)cocontinous $G \colon \R \C \to \D$, we have $(GJ)^\# \cong
    G$. But $(GJ)^\#(\phi) = \phi \star GJ \cong G(\phi \star J) \cong
    G(\phi)$ using cocontinuity of $G$ and density of $J$.\qedhere
  \end{enumerate}
\end{proof}
\begin{Rk}
  The notion of MacNeille completion is self-dual in that $\R\C =
  (\R(\C^\op))^\op$. It follows that $\R\C$ is also characterised by a
  dual universal property: it provides a bireflection of $\C$, seen as
  an object of the $2$-category $\cat{CCONTS}$ of $\Q$-categories and
  cut-continuous $\Q$-functors, onto the full sub-$2$-category
  $\cat{COTOT}$ whose objects are the cototal (=total)
  $\Q$-categories.
\end{Rk}

 \bibliographystyle{acm}

\bibliography{bibdata}
\end{document}